\documentclass[10pt]{amsart}
\usepackage{amsmath}
\usepackage{amssymb}
\usepackage{amsfonts}
\usepackage{pxfonts}
\usepackage{amsthm}
\usepackage{amsxtra}
\usepackage{epic,eepic}
\usepackage{epsfig}
\usepackage[dvips]{color}
\usepackage{graphics}
\usepackage{fontenc}
\usepackage{fullpage}
\usepackage{hyperref}

\usepackage{graphicx}
\usepackage{setspace}
\usepackage{fullpage}

\newcommand{\mathsym}[1]{{}}
\newcommand{\unicode}[1]{{}}

\theoremstyle{plain}

\newtheorem{lemma}{Lemma}
\newtheorem{corollary}{Corollary}
\newtheorem{proposition}{Proposition}

\theoremstyle{definition}

\theoremstyle{remark}
\newtheorem{remark}{Remark}

\newcommand{\C}{\mathbb R}

\begin{document}

\title{On the Variance of the Index for the Gaussian Unitary Ensemble}
\author{N. S. Witte and P.J. Forrester}
\address{Department of Mathematics and Statistics, University of Melbourne, Victoria 3010, Australia}
\email{\tt n.witte@ms.unimelb.edu.au, p.forrester@ms.unimelb.edu.au}

\begin{abstract}
We derive simple linear, inhomogeneous recurrences for the variance of the index by utilising 
the fact that the generating function for the distribution of the number of positive eigenvalues
of a Gaussian unitary ensemble is a $\tau$-function of the fourth Painlev\'e equation.
From this we deduce a simple summation formula, several integral representations and finally
an exact hypergeometric function evaluation for the variance. 
\end{abstract} 

\maketitle

\section{Introduction}

The exact characterisation of distribution functions in random matrix theory using integrable 
systems theory is very prevalent throughout random matrix theory. As well as bringing to the
fore rich mathematical structures, such characterisation has the practical consequence of both
providing efficient computational schemes for the distributions, and in the determination of
their asymptotic properties.

Here, inspired by a recent work of Majumdar et al \cite{MNSV_2011}, we will undertake a study
from the viewpoint of integrable systems of the so-called index distribution for the 
$ n\times n $ Gaussian unitary ensemble (GUE). The index distribution is a discrete probability
distribution for the number of positive eigenvalues. It is therefore specified by the 
probability $ p(n_{+},n) $ for there being $ n_{+} $ positive eigenvalues out of the $ n $
eigenvalues in total.
We refer to the Introduction of \cite{MNSV_2011} for an extended discussion as to the interest
in the index distribution for its relevance to studies in disordered systems,
landscape based string theory and even quantum cosmology.

Explicitly, our interest is in the integrable systems context of the variance of the index,
\begin{equation}
 \Delta_{n} := \langle \left( n_{+}-\langle n_{+} \rangle \right)^2 \rangle = \sum^{n}_{n_{+}=0} \left( n_{+}-\langle n_{+} \rangle \right)^2 p(n_{+},n) .
\end{equation}
Note that by the symmetry $ n_{+}\mapsto n-n_{+} $ we must have the average number of positive eigenvalues, 
$ \langle n_{+} \rangle = \frac{1}{2}n $.
Our working expression, see Eq. (137) of \cite{MNSV_2011}, is obtained by introducing the
 ``Poissonisation'' of the index distribution
\begin{equation}
 {Z}_{n}(s) := \sum^{n}_{n_{+}=0} e^{-sn_{+}} p(n_{+},n) .
\label{MGF}
\end{equation}
Most importantly, making use of the explicit form of the joint eigenvalue probability density
function for the GUE, this has the multiple-integral representation
\begin{equation}
 {Z}_{n}(s) = \frac{1}{N_{n}}\int^{\infty}_{-\infty} d\lambda_{1} \ldots \int^{\infty}_{-\infty} d\lambda_{n}\, 
                  e^{-\sum^{n}_{j=1}\lambda_{j}^2-s\sum^{n}_{j=1}\theta(\lambda_{j})}\prod_{1\leq j<k\leq n}(\lambda_{j}-\lambda_{k})^2 , \quad n\geq 1, \quad Z_{0}=1 ,
\label{MGF_integral}
\end{equation}
where
\begin{equation}
 \theta(\lambda) =
 \begin{cases}
   1, & \lambda > 0 \\
   \tfrac{1}{2}, & \lambda = 0 \\
   0, & \lambda < 0 
 \end{cases} .
\end{equation}
is the Heaviside step function. The normalisation $ N_{n} $ is the Gaussian weight form of the
Selberg integral (see Proposition 4.7.1 of \cite{rmt_Fo})
\begin{equation}
 N_{n} = \int^{\infty}_{-\infty} d\lambda_{1} \ldots \int^{\infty}_{-\infty} d\lambda_{n}\, 
                  e^{-\sum^{n}_{j=1}\lambda_{j}^2}\prod_{1\leq j<k\leq n}(\lambda_{j}-\lambda_{k})^2 ,
\end{equation}
and has the evaluation
\begin{equation}
  N_{n} = 2^{-\frac{1}{2}n(n-1)}\pi^{n/2}n!G(n+1) ,
\label{Norm}
\end{equation}
where $ G(z) $ is the Barnes $G$ function (see of 5.17 of \cite{DLMF}). From the generating polynomial
(\ref{MGF}) we can reclaim the probability itself and this is given by
\begin{equation}
 p(n_{+},n) = \frac{1}{N_{n}}{n \choose n_{+}}
              \int^{\infty}_{0} d\lambda_{1} \ldots \int^{\infty}_{0} d\lambda_{n_{+}}\,
              \int^{0}_{-\infty} d\lambda_{n_{+}+1} \ldots \int^{0}_{-\infty} d\lambda_{n}\, 
                  e^{-\sum^{n}_{j=1}\lambda_{j}^2}\prod_{1\leq j<k\leq n}(\lambda_{j}-\lambda_{k})^2 ,\quad 0\leq n_{+}\leq n,
\label{p_integral}
\end{equation}
Using the form (\ref{MGF}) we seek to compute the variance according to the formula
\begin{equation}
 \Delta_{n} = \left. \frac{1}{{Z}_{n}}\frac{d^2}{ds^2}{Z}_{n} \right|_{s=0}-\tfrac{1}{4}n^2 ,\quad n\geq 0.
\label{VoI}
\end{equation} 

From the integrable systems viewpoint, the significance of (\ref{MGF})
is that it can be written in the Hankel determinant form (see \cite{rmt_Fo}), Eq. (5.75))
\begin{equation}
 {Z}_{n}(s) = \frac{n!}{N_{n}}\det [ \int^{\infty}_{-\infty}d\lambda\,\lambda^{j+k}e^{-\lambda^2-s\theta(\lambda)} ]_{0\leq j,k\leq n-1} ,\quad n\geq 1.
\label{MGF_det}
\end{equation}
The work \cite{MNSV_2011} used (\ref{VoI}) to derive various summation and an integral formula, and 
to deduce the leading order asymptotics of $ \Delta_{n} $. However no tie in with integrable systems theory was made.  
In fact $ {Z}_{n}(s) $ is a $\tau$-function of the fourth Painlev\'e equation and much of the fundamental 
theory for this system has already been worked out in a study by the present authors \cite{FW_2001a}.
In particular simple and effective formulae to characterise the variance of the index
will be derived as an application of the results in this latter study. We also remark that the techniques
employed here can be generalised to treat the higher moments of the index distribution with a little more labour, 
and in fact can be applied to other $ \beta=2 $ ensembles where the $\tau$-function theory has been similarly developed.

\section{The GUE and the fourth Painlev\'e Equation}

\subsection{Okamoto $\tau$-function theory}
We recount some of the fundamental facts in Okamoto's $\tau$-function theory of the fourth
Painlev\'e equation \cite{Ok_1986},\cite{Ok_1981}, as given in the reformulation by Noumi and Yamada \cite{NY_1999} 
(see also the monograph \cite{Noumi_2004}). As mentioned in the Introduction this theory was summarised in 
\cite{FW_2001a} and applied to the Gaussian unitary ensembles, and it is the key results given there that 
will be of the greatest utility in addressing our problem.

The first key result is the symmetric form of the fourth Painlev\'e equation \cite{Ad_1994,NY_1998a,NY_2000}.
\begin{proposition}
The fourth Painlev{\'e} equation is equivalent to the coupled set of autonomous differential equations 
(where $' = d/dt$) 
\begin{equation}
\begin{split}
   f_0' & = f_0(f_1 - f_2) + 2 \alpha_0 \ ,\\
   f_1' & = f_1(f_2 - f_0) + 2 \alpha_1 \ ,\\
   f_2' & = f_2(f_0 - f_1) + 2 \alpha_2 \ ,
\end{split}
\label{c1}
\end{equation}
subject to the constraint
\begin{equation}\label{c2}
f_0 + f_1 + f_2 = 2t \ ,
\end{equation}
with the fourth Painlev\'e transcendant given by $y= - f_1$ and where the parameters $ \alpha_j \in \C $ with 
$ \alpha_0 + \alpha_1 + \alpha_2 = 1 $ are related to the conventional parameters 
\begin{equation}\label{PIVa}
\alpha = \alpha_0 - \alpha_2, \qquad \beta = - 2\alpha_1^2 \ .
\end{equation}
\end{proposition}
The parameters $ (\alpha_0,\alpha_1,\alpha_2) $ form a triangular co-ordinate system system
in the root system attached to the extended affine Weyl group $ A_{2}^{(1)} $ and the fundamental
Weyl chamber given by $ 0<\alpha_0,\alpha_1,\alpha_2 <1 $. 

The second feature of the fourth Painlev\'e system is the existence of a Hamiltonian
formulation, which is essential to our theory.
\begin{proposition}\cite{Ok_1986,KMNOY_2001}
The PIV dynamical system is a non-autonomous Hamiltonian system $ \{q,p;t,H\} $ with the Hamiltonian
\begin{align}
H  & = (2p - q - 2t)pq - 2 \alpha_1 p - \alpha_2 q ,\\
   & = \tfrac{1}{2} f_0 f_1 f_2 + \alpha_2 f_1 - \alpha_1 f_2  ,
\label{7.a}
\end{align}
and canonical variables $ q, p $
\begin{equation}\label{12'}
 q = -f_1, \qquad p = \tfrac{1}{2}f_2 \ .
\end{equation}
\end{proposition}

This allows us to define the $\tau$-function $\tau(t)$ in terms of the Hamiltonian $H(t)$ by
\begin{equation}\label{tau}
H(t) =: {d \over dt} \log \tau(t).
\end{equation}
We will construct a countable sequence of classical solutions to the fourth Painlev\'e equation
with parameters $ (\alpha_0+n, \alpha_1, \alpha_2-n ) $, indexed by $ n \in\mathbb{Z}_{\geq 0} $.
The seed solution, $ n=0 $, locates the parameters on one of the walls of the affine Weyl chamber,
$ \alpha_2=0 $, and this is one of the necessary conditions for a classical solution of the 
fourth Painlev\'e equation.
Thus we construct a sequence of $\tau$-functions $ \tau_{n}(t) $ (which were denoted by $ \tau_3[n](t;\alpha_1) $
in \cite{FW_2001a}) that also satisfy the Toda lattice equation
\begin{equation}\label{puk}
{d^2 \over dt^2}\log\left( e^{-t^2(\alpha_{2}-n)}\tau_{n} \right)
   = C{\tau_{n+1}\tau_{n-1} \over \tau^2_{n}} \ .
\end{equation}

Proposition 7 of \cite{FW_2001a} states, in part, that such $\tau$-functions are also Hankel
determinants.
\begin{proposition}
A sequence of solutions to the Toda lattice equation (\ref{puk}) with parameters 
$ (\alpha_0+n, \alpha_1, -n) $, $ n \geq 0 $, starting from the Weyl chamber wall $ \alpha_2 = 0 $ is 
given by the determinant form
\begin{equation}
 \tau_{n}(t;\alpha_1) = C_{n}\det \Big[ \int_{-\infty}^t (t-x)^{-\alpha_1+i+j} e^{-x^2} \, dx \Big]_{i,j=0,\dots,n-1} ,\quad n\geq 1, \quad \tau_{0}=1 ,
\end{equation}
where $ C_{n} $ is a constant independent of $ t $ and $ \alpha_1 $.
\end{proposition}
However this result can be easily generalised in a way that is of direct significance to our problem.
\begin{proposition}
A sequence of more general solutions to the Toda lattice equation (\ref{puk}) with parameters 
$ (\alpha_0+n, \alpha_1,-n) $, $ n \geq 0 $, starting from the Weyl chamber wall $ \alpha_2 = 0 $ is 
given by the determinant form
\begin{equation}
 \tau_{n}(t;\alpha_1;\xi) = C_{n}\det \Big[ \left( \int_{-\infty}^{t}+\xi\int_{t}^{\infty} \right) (t-x)^{-\alpha_1+i+j} e^{-x^2} \, dx \Big]_{i,j=0,\dots,n-1} ,
 \quad n\geq 1, \quad \tau_{0}=1 ,
\label{PIV_det}
\end{equation}
where $ \xi $ is an arbitrary complex parameter and $ C_{n} $ is also independent of $ \xi $.
\end{proposition}
\begin{proof}
We do not enter into the details of the proof of this as the proof of Proposition 7 in \cite{FW_2001a}
carries over in this case in its entirety.
\end{proof}

Finally we construct pure recurrence relations in $ n $ for our dynamical system which
can be shown to be equivalent to the second order difference equations of the first
discrete Painlev\'e equation ${\rm dP_{I}} $ (see Proposition 11 of \cite{FW_2001a}).
\begin{proposition}
The symmetric variables $ \{f_{j,n}(t;\alpha_1)\}_{j=0,1,2} $ satisfy a 
coupled system of first order recurrence relations in $ n\geq 0 $
\begin{align}
  f_{0,n} + f_{0,n-1} & = 2t - f_{2,n} + {2n \over f_{2,n}} ,\quad n\geq 1 ,
\label{fshift:a}\\
  f_{2,n+1} + f_{2,n} & = 2t - f_{0,n} + {2(n+\alpha_0) \over f_{0,n}} ,\quad n\geq 0 ,
\label{fshift:b}
\end{align}
together with $ f_{0,n}+f_{1,n}+f_{2,n}=2t $.
\end{proposition}

In \cite{FW_2001a} the symmetric variables were denoted by $ \{f_{j}[n]\}_{j=0,1,2} $. A number of
relations enable one to recover the symmetric variables from the $\tau$-functions and vice-versa. Thus one
has
\begin{equation}
  f_{1,n}f_{2,n} = -2n+\frac{\tau_{n+1}\tau_{n-1}}{\tau_{n}^2} ,\quad n\geq 1 ,
\label{fftau}
\end{equation}
which motivates the definition
\begin{equation}
  g_{n}(\xi) := f_{1,n}(\xi)f_{2,n}(\xi) .
\label{gDef}
\end{equation}
The symmetric variables can be recovered using
\begin{align}
 f_{1,n}^2 & = \frac{g_{n}{}^2}{g_{n}+2n}\frac{\left(g_{n}+2n\right)g_{n+1}+2ng_{n}}{\left(g_{n}+2n\right)g_{n-1}+2ng_{n}} ,
\\
 f_{2,n}^2 & = \left(g_{n}+2n\right)\frac{\left(g_{n}+2n\right)g_{n-1}+2ng_{n}}{\left(g_{n}+2n\right)g_{n+1}+2ng_{n}} ,
\end{align}
modulo the ambiguity associated with the choice of the branch of the square-root.

Also given in the work \cite{FW_2001a} is a recurrence relation for the $\tau$-function directly.
\begin{proposition}
A single recurrence relation for the $\tau$-functions is given in \cite{FW_2001a} by Eq. (2.82),
which states $ n\geq 2 $
\begin{multline}
 4t^2\left[ \tau_{n+1}\tau_{n-1}-2n\tau_{n}^2 \right]\left[ \tau_{n+1}\tau_{n-1}+2(\alpha_1-n)\tau_{n}^2 \right]\\
 \times\left[ \tau_{n+2}\tau_{n}\tau_{n-1}-2\tau_{n+1}^2\tau_{n-1}+4n(\alpha_1-n)\tau_{n+1}\tau_{n}^2 \right]
       \left[ \tau_{n-2}\tau_{n}\tau_{n+1}+2\tau_{n-1}^2\tau_{n+1}+4n(\alpha_1-n)\tau_{n-1}\tau_{n}^2 \right]\\
 = \Bigg\{ \tau_{n}^3\tau_{n+2}\tau_{n-2}+\tau_{n+2}\tau_{n-1}^2\left[ \tau_{n+1}\tau_{n-1}+(2\alpha_1+2-4n)\tau_{n}^2 \right]
                                      +\tau_{n+1}^2\tau_{n-2}\left[ \tau_{n+1}\tau_{n-1}+(2\alpha_1-2-4n)\tau_{n}^2 \right] \\
          -16n^2(\alpha_1-n)^2\tau_{n}^5+16n(\alpha_1-n)(\alpha_1-2n)\tau_{n}^3\tau_{n+1}\tau_{n-1}
                        -4(2n^2-2\alpha_1 n+1)\tau_{n}\tau_{n+1}^2\tau_{n-1}^2
   \Bigg\}^2 .
\label{FulltauR}
\end{multline}
\end{proposition}
\begin{proof}
In \cite{FW_2001a} the $ f_{1}, f_{2} $ system of coupled recurrences can be eliminated in favour of
the $\tau$-functions through the relation
\begin{equation}
 {4t^2c_{n}(b_{n-1}-c_{n})(b_{n}-c_{n}) = \left(b_{n}b_{n-1}-c_{n}^2\right)^2} ,
\end{equation} 
where the auxiliary variables are
\begin{equation}
  b_{n} := f_{1,n}f_{2,n}+f_{1,n+1}f_{2,n+1}+2\alpha_1 ,
\end{equation}
and
\begin{equation}
 c_{n} := f_{1,n}f_{2,n}\frac{f_{1,n}f_{2,n}+2\alpha_1}{f_{1,n}f_{2,n}+2n} .
\end{equation}
Eq. (\ref{FulltauR}) results from the substitution of (\ref{fftau}) into this system.
\end{proof}

Our first key result is the following identification.
\begin{proposition}\label{Identify}
The moment generating function (\ref{MGF}) is the $n$-th member of the classical $\tau$-function sequence 
of the fourth Painlev\'e equation with independent variable $ t=0 $, parameters 
$ (\alpha_0+n,\alpha_1,-n) $ with $ \alpha_0=1, \alpha_1=0 $ and initial conditions $ \xi = e^{-s} $ by
\begin{equation}
  {Z}_{n}(s) = \frac{n!}{2^{n(n-1)}N_{n}}\tau_{n}(0;0;e^{-s}) ,\quad n\geq 0,
\end{equation}
where $ N_{n} $ has been given by (\ref{Norm}).
With these specialisations the $\tau$-functions satisfy a simplified recurrence in $ n\geq 2 $
\begin{multline}
 \tau_{n+2}\tau_{n}^3\tau_{n-2}
+\tau_{n+2}\tau_{n-1}^2\left[ \tau_{n+1}\tau_{n-1}-(4n-2)\tau_{n}^2 \right]
+\tau_{n-2}\tau_{n+1}^2\left[ \tau_{n+1}\tau_{n-1}-(4n+2)\tau_{n}^2 \right]
\\
-4\left(2n^2+1\right)\tau_{n+1}^2\tau_{n}\tau_{n-1}^2
+32 n^3\tau_{n+1}\tau_{n}^3\tau_{n-1}-16 n^4\tau_{n}^5 = 0 .
\label{tauR}
\end{multline}
or alternatively the $ g$-variables satisfy the recurrence
\begin{equation}
  g_{n}^4 = (2 n+g_{n})^2 (g_{n}+g_{n-1}) (g_{n}+g_{n+1}) ,\quad n\geq 1.
\label{gR}
\end{equation}
\end{proposition}
\begin{proof}
This follows from a comparison of (\ref{MGF_det}) and (\ref{PIV_det}) and the only issue that remains
is to determine the prefactor $ C_{n} $. This can be done for $ \xi=1 $ as $ C_{n} $ only depends on $ n $
and not $ t, \xi $ etc. Thus we have
\begin{equation}
\tau _{n}(1)=C_{n}\Pi^{n-1}_{j=0}\pi^{1/2}2^{-j}j! = C_{n}2^{-\frac{1}{2}n(n-1)}\pi^{n/2}G(n+1) ,
\end{equation}
where $ G(z) $ is the Barnes $G$ function (see of 5.17 of \cite{DLMF}), which we substitute into 
(\ref{tauR}). Noting that this only involves the second-ratio of $ C_{n} $ we make the ansatz
\begin{equation}
 C_{n+2} = \mathit{c}\frac{C_{n+1}^2}{C_{n}} ,
\end{equation}
and compute that the recurrence contains an overall factor of $ (\mathit{c}-4)^2 $.
Thus we set $ \mathit{c}=4 $ and take the simple solution $ C_{n} = 2^{n(n-1)} $.
\end{proof}
Thus our key quantity of interest is now expressed by the $\tau$-functions through
\begin{equation}
 \Delta_{n}:=\left.\frac{\tau^{''}_{n}(\xi)}{\tau_{n}(\xi)}\right|_{\xi=1}-\frac{n^2}{4}+\frac{n}{2}, \quad n\geq 0 .
\end{equation}

\begin{remark}
Clearly $ \tau_{n}(\xi) $ is a polynomial in $ \xi $ with degree $ \deg_{\xi}=n $, however the independent
variable in our context is not a ``time'' variable in the conventional dynamics of the fourth Painlev\'e system
and does not appear in any of the characterising differential or difference equations. This parameter arises as a boundary 
or initial condition and will also be part of the monodromy data of the associated isomonodromic system.
Our system is a classical one and therefore the monodromy is expected to trivialise in some way, i.e. the
monodromy matrices will be upper or lower triangular or $ \pm 1 $ multiples of the identity, yet it will be a 
one-parameter system through $ \xi $. 
\end{remark}

\begin{remark}
The seed solution is located at the corner of the Weyl chamber $ \alpha_0=1, \alpha_1=0, \alpha_2=0 $,
and this is known as one of the conditions for the rational solution that generate the family of ``generalised
Hermite polynomials'' \cite{NY_1999}, \cite{Cl_2003}, \cite{Cl_2009}, and are expressible as Schur functions for rectangular Young diagrams.
However it is not clear there is any link between those polynomials and the ones treated here.
\end{remark}

\begin{remark}
In our identification, Prop. \ref{Identify}, we set $ \alpha_1=0 $, however all of the analysis we
will present carries through with $ \alpha_1 $ non-zero. The case where such a situation is of
relevance to our original problem is when $ \alpha_1 $ is a negative integer and we are 
conditioning a fixed number of eigenvalues to lie at the origin. 
\end{remark}

\subsection{Low $ n $ Solutions}

We present some data for the low $ n $ examples of our system, in order to furnish initial values 
in the $ n $ recurrences we are about to employ and also to furnish independent checks on our results.

With the initial values for the $\tau$-functions $ \tau_{-1}(\xi) := 0 $, $ \tau_{0}(\xi):=1 $ 
and their Hankel determinant form for $ n\geq 1 $
\begin{align}
\tau_{n}(\xi) & = 2^{n(n-1)}\det\left[\int^{0}_{-\infty}dx\,(-x)^{i+j}e^{-x^2}+\xi\int_{0}^{\infty}dx\,(-x)^{i+j}e^{-x^2}\right]_{0 \leq i,j \leq n-1} ,
\\
              & = 2^{n(n-2)}\det\left[ (\xi+(-1)^{i+j})\Gamma(\tfrac{1}{2}(i+j+1))\right]_{0 \leq i,j \leq n-1} ,
\end{align}
we easily compute the first few cases of the $\tau$-functions
\begin{gather*}
 \tau_{1}(\xi) = \frac{1}{2} \sqrt{\pi } (\xi+1) ,
\\
\tau_{2}(\xi) = \frac{1}{2} \left[ \pi(\xi+1)^2-2(\xi-1)^2 \right] ,
\\
\tau_{3}(\xi) = 2\sqrt{\pi}(\xi+1)\left[ \pi(\xi+1)^2-3(\xi-1)^2 \right] ,
\\
\tau_{4}(\xi) = 8\left[ 6\pi^2(\xi+1)^4-29\pi(\xi^2-1)^2+32 (\xi-1)^4 \right] ,
\\
 \tau_{5}(\xi) = 128\sqrt{\pi}(\xi+1)\left[ 72\pi^2(\xi+1)^4-435\pi(\xi^2-1)^2+656(\xi-1)^4 \right] ,
\\
 \tau_{6}(\xi) = 4096\left[ 4320\pi^3(\xi+1)^6-33723\pi^2(\xi-1)^2(\xi+1)^4+84168\pi(\xi-1)^4(\xi+1)^2-65536(\xi-1)^6 \right] .
\end{gather*} 
Correspondingly we compute the initial auxiliary variable $ g_{n} $ as 
\begin{gather*}
g_{0}(\xi) = 0
\\
g_{1}(\xi) =-\frac{4(\xi-1)^2}{\pi(\xi+1)^2} ,
\\
g_{2}(\xi) = 4(\xi-1)^2\frac{\left[ \pi(\xi+1)^2-4(\xi-1)^2 \right]}{\left[ \pi(\xi+1)^2-2(\xi-1)^2 \right]^2} ,
\\
g_{3}(\xi) =-\frac{(\xi-1)^2}{\pi(\xi+1)^2}\frac{\left[ 5\pi^2(\xi+1)^4-36\pi(\xi^2-1)^2+64(\xi-1)^4 \right]}{\left[ \pi(\xi+1)^2-3(\xi-1)^2 \right]^2} .
\end{gather*}
The initial values of the symmetric variable $ f_{1,n} $ and the leading terms of their Laurent expansions are
\begin{align}
 f_{1,0}(\xi) & =\frac{2 (\xi-1)}{\sqrt{\pi}(\xi+1)} ,
\\
  & =\frac{\xi-1}{\sqrt{\pi }}-\frac{(\xi -1)^2}{2 \sqrt{\pi }}+\frac{(\xi -1)^3}{4 \sqrt{\pi }}+O(\xi-1)^4 ,
\end{align}
\begin{align}
 f_{1,1}(\xi) & =\frac{4(\xi-1)^3}{\surd\pi(\xi+1)\left[ \pi(\xi+1)^2-2 (\xi-1)^2 \right]} ,
\\
 & =\frac{(\xi-1)^3}{2\pi^{3/2}}-\frac{3(\xi-1)^4}{4\pi^{3/2}}+\frac{(3\pi+1)(\xi-1)^5}{4\pi^{5/2}}+O(\xi-1)^{6} ,
\end{align}
\begin{align}
f_{1,2}(\xi) & =\frac{(\xi-1)}{\surd\pi(\xi+1) }\frac{\left[ \pi(\xi+1)^2-4(\xi-1)^2 \right]^2}{\left[ \pi(\xi+1)^2-3(\xi-1)^2 \right]\left[ \pi(\xi+1)^2-2(\xi-1)^2 \right]} ,
\\
 & = \frac{\xi-1}{2\sqrt{\pi}}-\frac{(\xi-1)^2}{4\sqrt{\pi}}+\frac{(\pi-3)(\xi-1)^3}{8\pi^{3/2}}+O(\xi-1)^4 .
\end{align}
Those of $ f_{2,n} $ are
\begin{equation}
{f_{2,0}(\xi)=0} ,
\end{equation}
\begin{align}
f_{2,1}(\xi) & =-\frac{\left[ \pi(\xi+1)^2-2(\xi-1)^2 \right]}{\surd \pi  \left(\xi ^2-1\right)} ,
\\
 & =-\frac{2\sqrt{\pi}}{\xi-1}-\sqrt{\pi}+\frac{\xi-1}{\sqrt{\pi}}+O(\xi-1)^2 ,
\end{align}
\begin{align}
f_{2,2}(\xi) & =\frac{4\surd\pi\left[ \xi^2-1 \right]\left[ \pi(\xi +1)^2-3(\xi-1)^2 \right]}{\left[ \pi(\xi+1)^2-4(\xi-1)^2 \right]\left[ \pi(\xi+1)^2-2(\xi-1)^2 \right]} ,
\\
 & =\frac{2(\xi-1)}{\sqrt{\pi}}-\frac{(\xi-1)^2}{\sqrt{\pi}}+\frac{(\pi+3)(\xi-1)^3}{2\pi^{3/2}}+O(\xi-1)^4 ,
\end{align}
whilst the corresponding data of $ f_{0,n} $ are
\begin{align}
f_{0,0}(\xi) & =-\frac{2(\xi-1)}{\sqrt{\pi}(\xi+1)} ,
\\
 & =-\frac{\xi-1}{\sqrt{\pi}}+\frac{(\xi-1)^2}{2\sqrt{\pi}}-\frac{(\xi-1)^3}{4\sqrt{\pi}}+O(\xi-1)^4 ,
\end{align}
\begin{align}
f_{0,1}(\xi) &  = \frac{\sqrt{\pi}(\xi+1)\left[ \pi(\xi+1)^2-4(\xi-1)^2 \right]}{(\xi-1)\left[ \pi(\xi+1)^2-2(\xi-1)^2 \right]} ,
\\
 & = \frac{2\sqrt{\pi}}{\xi-1}+\sqrt{\pi}-\frac{\xi-1}{\sqrt{\pi}}+O(\xi-1)^2 ,
\end{align}
\begin{align}
f_{0,2}(\xi)  & =
-\frac{(\xi-1)\left[ \pi(\xi+1)^2-2(\xi-1)^2 \right]\left[ 5\pi(\xi+1)^2-16(\xi-1)^2 \right]}
      {\sqrt{\pi}(\xi+1)\left[ \pi(\xi+1)^2-4(\xi-1)^2 \right]\left[ \pi(\xi+1)^2-3(\xi-1)^2 \right]} ,
\\
 & =-\frac{5(\xi-1)}{2\sqrt{\pi}}+\frac{5(\xi-1)^2}{4\sqrt{\pi}}-\frac{(5\pi+9)(\xi-1)^3}{8\pi^{3/2}}+O(\xi-1)^4 .
\end{align}
We remarked earlier that there is a sign ambiguity in deducing the symmetric variables from the
$\tau$-functions.
In computing the above variables we have adopted a particular and arbitrary choice for the sign of an initial
variable, say that of $ f_{1,0} $ as being $ +1 $. The recurrence system (\ref{fshift:a},\ref{fshift:b})
fixes the sign of the other components at $ n=0 $ and all components subsequently for $ n>0 $. 

Finally we give some initial values of the variance of the index itself
\begin{align*}
 \Delta_{1} & =  \frac{1}{4},\\
 \Delta_{2} & =  \frac{1}{2}-\frac{1}{\pi },\\
 \Delta_{3} & =  \frac{3}{4}-\frac{3}{2\pi},\\
 \Delta_{4} & =  1-\frac{29}{12 \pi },\\
 \Delta_{5} & = \frac{5}{4}-\frac{145}{48 \pi },\\
 \Delta_{6} & =  \frac{3}{2}-\frac{1249}{320 \pi },\\
            & \vdots \\
\Delta_{18} & =  \frac{9}{2}-\frac{1198597830455957}{91359323095040 \pi },\\
\Delta_{19} & = \frac{19}{4}-\frac{22773358778663183}{1644467815710720 \pi },\\
\Delta_{20} & =  5-\frac{183365193212828149}{12497955399401472 \pi } .
\end{align*}

\section{Linear recurrences for derivatives}

We recall that the full non-linear recurrence relation for the $\tau$-function, $ n\geq 2 $, is
\begin{multline}
 \tau_{n+2}\tau_{n-2}\tau_{n}^3+\tau_{n+2}\tau_{n-1}^2 \left[(2-4 n) \tau_{n}^2+\tau_{n-1} \tau_{n+1}\right] 
+\tau_{n-2} \tau_{n+1}^2 \left[(-2-4 n) \tau_{n}^2+\tau_{n-1} \tau_{n+1}\right]
\\
-16 n^4 \tau_{n}^5+32 n^3 \tau_{n-1} \tau_{n}^3 \tau_{n+1}-4\left(1+2 n^2\right) \tau_{n-1}^2 \tau_{n} \tau_{n+1}^2 = 0 .
\label{DiffEqn}
\end{multline}
where the dependence on $ \xi $ is completely implicit, through the initial conditions. Our 
strategy is to successively differentiate this relation with respect to $ \xi $ and set $ \xi\to 1 $
afterwards, utilising the results of earlier solutions at each step. This does not necessarily
guarantee a linear difference equation but in fact we will find that this is the case.

\begin{proposition}
The ratio 
\begin{equation}
  X_{n} = \frac{\tau^{'}_{n}(1)}{\tau _{n}(1)} ,\quad n\geq 0 ,
\label{XDef}
\end{equation}
satisfies the linear recurrence relation
\begin{equation}
  n X_{n+1}-(n+1)X_{n}-(n-2)X_{n-1}+(n-1)X_{n-2} = 0 , \quad n\geq 2,
\label{FirstDer}
\end{equation}
which has the solution
\begin{equation}
   X_{n} =  \frac{n}{2} ,\quad n\geq 0 ,
\label{FirstSoln}
\end{equation} 
given the initial conditions $ X_{0} = 0, X_{1} =\tfrac{1}{2}, X_{2} = 1 $.
This implies the assertion given in the introduction that
\begin{equation}
 \langle n_{+} \rangle = \left.\frac{1}{\tau_{n}}\frac{d}{d\xi}\tau_{n} \right|_{\xi=1} = \tfrac{1}{2}n ,\quad n\geq 0 .
\end{equation} 
\end{proposition}
\begin{proof}
We first set 
\begin{equation}
\tau _{n}(1) = 2^{n(n-1)/2}\pi^{n/2}G(n+1) ,\quad n\geq 0 ,
\end{equation}
and define
\begin{equation}
\tau^{'}_{n}(1) = 2^{n(n-1)/2}\pi^{n/2}G(n+1)X_{n} ,\quad n\geq 0 .
\end{equation}
Differentiating (\ref{DiffEqn}) once and setting both $ \tau_{n}(1) $ and $ \tau^{'}_{n}(1) $ as given above
leads to the result that the left-hand side is identically zero. Differentiating again and making the same
substitutions gives the result that the coefficients of the second derivatives vanish identically however
the first derivatives remain. After simplifying we find the relevant remnant factorises into a product 
of linear, difference equations 
\begin{multline}
  \left[ (n-1) X_{n-2}-(n-2) X_{n-1}-(n+1) X_{n}+n X_{n+1} \right]\\
\times \left[ n X_{n-1}-(n-1) X_{n}-(n+2) X_{n+1}+(n+1)X_{n+2} \right] = 0, \quad n\geq 2 .
\end{multline}
The second factor is the just the first under the replacement $ n\to n+1 $. A solution to the above 
difference equation could be
one that satisfies the third order linear homogeneous difference equation (\ref{FirstDer}) for either the 
odd $ n $ or even $ n $ cases, or possibly both. The latter case applies here because the initial conditions enforce this.
\end{proof}

Continuing with this procedure we have the following result.
\begin{proposition}
The variance of the index satisfies the third order inhomogeneous linear difference equation
\begin{equation}
  n \Delta_{n+1}-(n+1)\Delta_{n}-(n-2)\Delta_{n-1}+(n-1)\Delta_{n-2} = \frac{1-(-1)^n}{2}e_{n} ,\quad n \geq 2,
\label{auxD}
\end{equation} 
with $ e_{n} $ as yet undetermined.
\end{proposition}
\begin{proof}
When we differentiate (\ref{DiffEqn}) three times and make the previous substitutions for the $\tau$-function 
and its first derivative (including the solution given in (\ref{FirstSoln})) we find this identically vanishes. Let us set
\begin{equation}
 \tau^{''}_{n}(1) = 2^{n(n-1)/2}\pi^{n/2}G(n+1)Y_{n} ,\quad n\geq 0 .
\label{YDef}
\end{equation}
Differentiating once more and making all of the above substitutions we can factorise the result into the product of
two linear inhomogeneous difference equations
\begin{multline}
  \left[ -2n+1+2(n-1)Y_{n-2}-2(n-2)Y_{n-1}-2(n+1)Y_{n}+2nY_{n+1} \right] \\
\times \left[ -2n-1+2nY_{n-1}-2(n-1)Y_{n}-2(n+2)Y_{n+1}+2(n+1)Y_{n+2} \right] = 0 ,\quad n\geq 2 .
\end{multline}
Again the second factor is just a copy of the first under $ n\to n+1 $. However in this case there is an odd-even
dichotomy and therefore only one of the factors vanishes - the other is an as yet undetermined $n$-dependent term
which cannot be fixed by the recurrence relation (\ref{DiffEqn}). It turns out that the even case vanishes and the
odd is non-zero. Employing the change of variables
\begin{equation}
  Y_{n} = \frac{1}{4}n(n-2)+\Delta_{n} ,\quad n\geq 0 ,
\label{YDelta}
\end{equation}
we have (\ref{auxD}).
\end{proof}

We can also carry out the successive differentiation of the $ g=f_1\times f_2 $ recurrences (\ref{gR})
and deduce a system of linearised recurrences.
\begin{proposition}
The first two Taylor coefficients of $ g_{n} $ for all $ n\geq 0 $ at $ \xi=1 $ are given by
\begin{equation}
 g_{n}(1) = 0 ,
\end{equation} 
and
\begin{equation}
 g^{'}_{n}(1) = 0 ,
\end{equation}
whereas the second derivative is not uniquely determined by this method applied to (\ref{gR}). 
\end{proposition}
\begin{proof}
In a pattern similar to the case of the $\tau$-function recurrences the odd derivatives vanish
identically in a sequential way. The second derivative, for $ n>0 $, gives 
\begin{equation}
  \left[g^{'}_{n-1}(1)+g^{'}_{n}(1)\right] \left[g^{'}_{n}(1)+g^{'}_{n+1}(1)\right] = 0 .
\end{equation} 
The solution relevant here has $ g^{'}_{n}(1)=g^{'}_{n-1}(1)=0 $ for all $ n\geq 0 $. 
The fourth derivative gives a very similar result, namely
\begin{equation}
  \left[g^{''}_{n-1}(1)+g^{''}_{n}(1)\right] \left[g^{''}_{n}(1)+g^{''}_{n+1}(1)\right] = 0 .
\end{equation}
Subsequently we will see that $ g^{''}_{2m+1}(1)+g^{''}_{2m+2}(1) = 0 $, however $ g^{''}_{2m}(1)+g^{''}_{2m+1}(1) \neq 0 $.
\end{proof}

\begin{lemma}
The inhomogeneous term is given by
\begin{equation}
 e_{2m+1} = \tfrac{1}{2}g^{''}_{2m}(1)+\tfrac{1}{2}g^{''}_{2m+1}(1) ,\quad m\geq 0.
\label{egReln}
\end{equation} 
\end{lemma}
\begin{proof}
If we successively differentiate the definition (\ref{gDef}) combined with the relation (\ref{fftau})
and set $ \xi=1 $ we find, at the zeroth level
\begin{equation}
  g_{n}(1) = 0 .
\end{equation}
At the next level we employ the definition (\ref{XDef}) and its solution (\ref{FirstSoln}) to
deduce that
\begin{equation}
  g^{'}_{n}(1) = 0 .
\end{equation} 
This is consistent with our previous findings.
At the second level we utilise (\ref{YDef}) and (\ref{YDelta}) to derive
\begin{equation}
 g^{''}_{n}(1) = 2n\left[ \Delta_{n+1}+\Delta_{n-1}-2\Delta_{n} \right] ,\quad n\geq 1 .
\end{equation}
Then (\ref{egReln}) follows immediately from (\ref{auxD}).
\end{proof}

The even and odd cases of $ \Delta_{n} $ are related in our next result. 
\begin{corollary}
The odd and even cases of the variance of the index are related by
\begin{equation}
    \Delta_{2m+1} = \frac{2m+1}{2m}\Delta_{2m} ,\quad m\geq 1.
\label{even-odd}
\end{equation}
\end{corollary}
\begin{proof}
We observe that the difference equation (\ref{auxD}) can be written as
\begin{equation}
  r_{n}-r_{n-2} = \frac{1-(-1)^n}{2}e_{n} ,\quad n\geq 2 ,
\label{auxR}
\end{equation} 
where $ r_{n} := n\Delta_{n+1}-(n+1)\Delta_{n} $. In the even case of (\ref{auxR}) we have
$ r_{2m} = r_{2m-2} = \ldots = r_{0} = 0 $, and the relation (\ref{even-odd}) follows.
\end{proof}

Our recurrences have a structure allowing for a simple solution by discrete quadratures.
\begin{corollary}
The variance of the index has the unique solution
\begin{equation}
    \Delta_{n} = \tfrac{1}{4}n+n\sum^{n-1}_{j=1}\frac{1}{j(j+1)}\frac{1-(-1)^j}{2}\left( \sum^{(j-1)/2}_{l=1}e_{2l+1}-\frac{1}{\pi} \right) ,\quad n\geq 0.
\label{auxS}
\end{equation}
\end{corollary}
\begin{proof}
The difference equation (\ref{auxR}) can be solved by summation
\begin{equation}
  r_{n} = \frac{1-(-1)^n}{2}\left( \sum^{(n-1)/2}_{j=1}e_{2j+1}+r_{1} \right) ,\quad n\geq 1 ,         
\end{equation}
and we note the definition of $ r_{n} $ can be written in a telescoping form
\begin{equation}
  \frac{\Delta_{n+1}}{n+1}-\frac{\Delta_{n}}{n} = \frac{r_{n}}{n(n+1)} ,
\end{equation} 
and is also immediately soluble. In addition we utilise $ r_{1} = -1/\pi $.
\end{proof}

In order to evaluate the inhomogeneous term $ e_{n} $ we need to analyse the GUE $ f_0,f_2 $ coupled recurrences.
\begin{corollary}
The inhomogeneous term $ e_{n} $ is given by
\begin{equation}
 e_{2m+1} = -\frac{1}{\pi^2}\frac{\Gamma^2(m+\tfrac{1}{2})}{\Gamma^2(m+1)} ,\quad m\geq 0.
\label{auxF}
\end{equation} 
\end{corollary}
\begin{proof}
We seek an expansion of GUE $ f_0,f_2 $ coupled recurrences separated into the even-odd sub-cases
\begin{gather}
 f_{0,2m}(\xi)+f_{0,2m-1}(\xi) = -f_{2,2m}(\xi)+\frac{4 m}{f_{2,2m}(\xi)} ,\quad m\geq 1 ,
\label{f02R:a}\\
 f_{0,2m+1}(\xi)+f_{0,2m}(\xi) = -f_{2,2m+1}(\xi)+\frac{2(2m+1)}{f_{2,2m+1}(\xi)} ,\quad m\geq 0 ,
\label{f02R:b}\\
 f_{2,2m+1}(\xi)+f_{2,2m}(\xi) =  -f_{0,2m}(\xi)+\frac{2(2m+1)}{f_{0,2m}(\xi)} ,\quad m\geq 0 ,
\label{f02R:c}\\
 f_{2,2m+2}(\xi)+f_{2,2m+1}(\xi) = -f_{0,2m+1}(\xi)+\frac{2(2m+2)}{f_{0,2m+1}(\xi)} ,\quad m\geq 0 ,
\label{f02R:d}\\
 f_{0,2m}(\xi)+f_{1,2m}(\xi)+f_{2,2m}(\xi) = 0 ,\quad m\geq 0 ,
\label{f02R:e}\\
 f_{0,2m+1}(\xi)+f_{1,2m+1}(\xi)+f_{2,2m+1}(\xi) = 0 ,\quad m\geq 0 .
\label{f02R:f}
\end{gather}
in Laurent series about $\xi = 1 $, $ m\geq 0 $
\begin{gather}
  f_{2,2m}(\xi) =  \sum^{\infty}_{j=1}\mathcal{A}_{j,m}(\xi-1)^j,\\
  f_{2,2m+1}(\xi) =  \sum^{\infty}_{j=-1}\mathcal{B}_{j,m}(\xi-1)^j,\\
  f_{0,2m}(\xi) =  \sum^{\infty}_{j=1}\mathcal{C}_{j,m}(\xi-1)^j,\\
  f_{0,2m+1}(\xi) =  \sum^{\infty}_{j=-1}\mathcal{D}_{j,m}(\xi-1)^j,\\
  f_{1,2m}(\xi) =  \sum^{\infty}_{j=1}\mathcal{E}_{j,m}(\xi-1)^j,\\
  f_{1,2m+1}(\xi) =  \sum^{\infty}_{j=3}\mathcal{F}_{j,m}(\xi-1)^j .
\end{gather}
By solving for the leading coefficients of these expansions we can determine the inhomogeneous term because
\begin{align}
  \tfrac{1}{2}g^{''}_{2m}(1) & = \mathcal{E}_{1,m}\mathcal{A}_{1,m} ,
\label{gCff:a}\\
  \tfrac{1}{2}g^{''}_{2m+1}(1) & = \mathcal{F}_{3,m}\mathcal{B}_{-1,m} ,
\label{gCff:b}
\end{align}
for $ m\geq 0 $.

Our solution schema is as follows. We first note that coefficient of $ (\xi-1)^{-1} $ of (\ref{f02R:a}) gives
\begin{equation}
  \mathcal{D}_{-1,m-1} = \frac{4 m}{\mathcal{A}_{1,m}} ,\quad m\geq 1 .
\label{AUX:z}
\end{equation}
Now the coefficient of $ (\xi-1)^{0} $ of (\ref{f02R:a}) gives
\begin{equation}
  \frac{4 m \mathcal{A}_{2,m}}{\mathcal{A}_{1,m}^2} = -\mathcal{D}_{0,m-1} ,\quad m\geq 1 ,
\end{equation}
which combined with (\ref{AUX:z}) yields
\begin{equation}
  \mathcal{D}_{0,m} = -4(m+1)\frac{\mathcal{A}_{2,m+1}}{\mathcal{A}_{1,m+1}^2} ,\quad m\geq 0 .
\label{AUX:a}
\end{equation}
Next we combine the coefficient of $ (\xi-1)^{-1} $ of (\ref{f02R:f}) with the coefficient of $ (\xi-1)^{-1} $ of (\ref{f02R:c})
to deduce the relations
\begin{equation}
  \mathcal{D}_{-1,m} = -\frac{4m+2}{\mathcal{C}_{1,m}} ,\quad m\geq 0 ,
\label{AUX:bp}
\end{equation}
and
\begin{equation}
  \mathcal{B}_{-1,m} = \frac{4m+2}{\mathcal{C}_{1,m}} ,\quad m\geq 0 .
\label{AUX:b}
\end{equation}
Continuing we combine (\ref{AUX:z}) and (\ref{AUX:b}) to derive
\begin{equation}
  \mathcal{C}_{1,m} = -\frac{2m+1}{2m+2}\mathcal{A}_{1,m+1} ,\quad m\geq 0 .
\label{AUX:c}
\end{equation} 
Now we use the coefficient of $ (\xi-1)^{0} $ of (\ref{f02R:c}) to give
\begin{equation}
   \mathcal{C}_{2,m} = -\frac{1}{4m+2}\mathcal{B}_{0,m}\mathcal{C}_{1,m}^2 ,\quad m\geq 0 .
\label{AUX:d}
\end{equation}
Proceeding we employ coefficient of $ (\xi-1)^{1} $ of (\ref{f02R:c}) to give
\begin{equation}
  \mathcal{B}_{1,m} = -\mathcal{A}_{1,m}-\mathcal{C}_{1,m}+\frac{2(2m+1)}{\mathcal{C}_{1,m}^3}\left[ \mathcal{C}_{2,m}^2-\mathcal{C}_{1,m}\mathcal{C}_{3,m} \right] ,\quad m\geq 0 .
\label{AUX:e}
\end{equation}
Proceeding further the coefficient of $ (\xi-1)^{1} $ of (\ref{f02R:a}) under $ m\to m+1 $ yields 
\begin{equation}
  \mathcal{D}_{1,m} = -\mathcal{A}_{1,m+1}-\mathcal{C}_{1,m+1}+\frac{4(m+1)}{\mathcal{A}_{1,m+1}^3}\left[ \mathcal{A}_{2,m+1}^2-\mathcal{A}_{1,m+1}\mathcal{A}_{3,m+1} \right] ,\quad m\geq 0 .
\label{AUX:f}
\end{equation}
Now we employ the result for $ \mathcal{D}_{-1,m} $ from (\ref{AUX:z}), $ \mathcal{D}_{1,m} $ from 
(\ref{AUX:f}) and $ \mathcal{D}_{0,m} $ from (\ref{AUX:a}) to compute that
\begin{equation}
 \mathcal{D}_{-1,m}\mathcal{D}_{1,m}-\mathcal{D}_{0,m}^2 
 = -4(m+1)\frac{\mathcal{A}_{1,m+1}+\mathcal{C}_{1,m+1}}{\mathcal{A}_{1,m+1}}-16(m+1)^2\frac{\mathcal{A}_{3,m+1}}{\mathcal{A}_{1,m+1}^3} ,\quad m\geq 0 .
\label{AUX:g}
\end{equation}
Into the coefficient of $ (\xi-1)^{3} $ of (\ref{f02R:d}) we employ the results of (\ref{AUX:g}) and 
(\ref{AUX:z}) and compute 
\begin{equation}
  \mathcal{B}_{3,m}+\mathcal{D}_{3,m} = \frac{1}{4(m+1)}\mathcal{A}_{1,m+1}^2\left[ \mathcal{A}_{1,m+1}+\mathcal{C}_{1,m+1} \right] ,\quad m\geq 0 .
\label{AUX:h}
\end{equation}
On the other hand if we take the coefficient of $ (\xi-1)^{3} $ of (\ref{f02R:b}) 
\begin{equation}
  \mathcal{B}_{3,m}+\mathcal{D}_{3,m}+\mathcal{C}_{3,m}+\frac{2(2m+1)}{\mathcal{B}_{-1,m}{}^3}\left[ -\mathcal{B}_{0,m}^2+\mathcal{B}_{-1,m}\mathcal{B}_{1,m} \right] = 0 ,\quad m\geq 0 ,
\end{equation}
and use (\ref{AUX:d}) for $ \mathcal{B}_{0,m} $, (\ref{AUX:bp}) for $ \mathcal{B}_{-1,m} $ and (\ref{AUX:e}) 
for $ \mathcal{B}_{1,m} $ we arrive at
\begin{equation}
  \mathcal{B}_{3,m}+\mathcal{D}_{3,m} = \frac{\mathcal{C}_{1,m}^2}{2(2m+1)}\left[ \mathcal{A}_{1,m}+\mathcal{C}_{1,m} \right] ,\quad m\geq 0 .
\label{AUX:i}
\end{equation}
Combining (\ref{AUX:i}) and (\ref{AUX:h}), with the assistance of (\ref{AUX:c}), gives us a first order 
difference equation
\begin{equation}
  \mathcal{A}_{1,m+1}+\mathcal{C}_{1,m+1} = \frac{2m+1}{2m+2}\left[ \mathcal{A}_{1,m}+\mathcal{C}_{1,m} \right] ,\quad m\geq 0 .
\end{equation}  
Using the initial data $ \mathcal{A}_{1,0}+\mathcal{C}_{1,0} = -\pi^{-1/2} $ this solved by
\begin{equation}
  \mathcal{A}_{1,m}+\mathcal{C}_{1,m} = -\frac{\Gamma(m+\tfrac{1}{2})}{\pi\Gamma(m+1)} ,\quad m\geq 0 .
\end{equation} 
Employing (\ref{AUX:c}) once more gives us the first order inhomogeneous difference equation
\begin{equation}
  \mathcal{A}_{1,m}-\frac{2m+1}{2m+2}\mathcal{A}_{1,m+1} = -\frac{\Gamma(m+\tfrac{1}{2})}{\pi\Gamma(m+1)} ,\quad m\geq 0 ,
\end{equation} 
which is soluble by multiplying through by a summation factor and the end result ($ \mathcal{A}_{1,0}=0 $) is
\begin{equation}
   \mathcal{A}_{1,m} = \frac{\Gamma(m+1)}{\pi\Gamma(m+\tfrac{1}{2})}\sum^{m-1}_{r=0} \frac{\Gamma^2(r+\tfrac{1}{2})}{\Gamma^2(r+1)} ,\quad m\geq 0 .
\end{equation} 
Back substituting into (\ref{AUX:z},\ref{AUX:bp},\ref{AUX:c}), the coefficient of $ (\xi-1)^{1} $ of (\ref{f02R:e}) 
and the coefficient of $ (\xi-1)^{3} $ of (\ref{f02R:f}) gives for $ m\geq 0 $, respectively
\begin{equation}
  \mathcal{D}_{-1,m} = \frac{4(m+1)}{\mathcal{A}_{1,m+1}} ,
\end{equation} 
\begin{equation}
  \mathcal{B}_{-1,m} = -\frac{4(m+1)}{\mathcal{A}_{1,m+1}} ,
\end{equation} 
\begin{equation}
  \mathcal{C}_{1,m} = -\frac{2m+1}{2m+2}\mathcal{A}_{1,m+1} ,
\end{equation} 
\begin{equation}
   \mathcal{E}_{1,m} = \frac{\Gamma(m+\tfrac{1}{2})}{\pi\Gamma(m+1)} ,
\end{equation}
and
\begin{equation}
   \mathcal{F}_{3,m} = \frac{1}{4\pi}\frac{\Gamma(m+\tfrac{3}{2})}{(m+1)\Gamma(m+2)}\mathcal{A}^2_{1,m+1} .
\end{equation}
Using these formulae and the relations (\ref{gCff:a},\ref{gCff:b}) and (\ref{egReln}) gives the 
evaluation (\ref{auxF}).

\end{proof}

By drawing together all the previous results we present the simplest recurrence system for the variance of
the index.
\begin{corollary}
The even and odd sub-cases of the variance of the index satisfy second-order, linear inhomogeneous
difference equations
\begin{equation}
  \frac{(2m+1)}{(2m-1)}\Delta_{2m+2}-\left[1+\frac{(2m+2)(2m+1)}{(2m)(2m-1)}\right]\Delta_{2m}+\frac{(2m)}{(2m-2)}\Delta_{2m-2}
  = -\frac{1}{2\pi^2}\frac{\Gamma(m+\tfrac{1}{2})\Gamma(m-\tfrac{1}{2})}{\Gamma^2(m+1)} ,\quad m\geq 1,
\label{HDE:a}
\end{equation}
and
\begin{equation}
  \frac{(2m+1)}{(2m+3)}\Delta_{2m+3}-\left[1+\frac{(2m)(2m-1)}{(2m+2)(2m+1)}\right]\Delta_{2m+1}+\frac{(2m)}{(2m+2)}\Delta_{2m-1}
  = -\frac{1}{2\pi^2}\frac{\Gamma^2(m+\tfrac{1}{2})}{\Gamma(m+2)\Gamma(m+1)} ,\quad m\geq 1.
\label{HDE:b}
\end{equation}  
respectively.
\end{corollary}

\begin{remark}
An alternative system of recurrences to that of (\ref{fshift:a},\ref{fshift:b}) was discussed in \cite{FW_2001a}
and was necessary in order to relate the system to the first discrete Painlev\'e equation. This system
is formulated in terms of $ \chi_{2n+1} := f_{2,n} $ and $ \chi_{2n+2} := f_{0,n} $ and the system of 
recurrences governing these are, for our specialisation 
\begin{gather}
\chi_{2m}(\xi)\left[ \chi_{2m+1}(\xi)+\chi_{2m}(\xi)+\chi_{2m-1}(\xi) \right] = 2m ,\quad m\geq 1 ,
\\
\chi_{2m+1}(\xi)\left[ \chi_{2m+2}(\xi)+\chi_{2m+1}(\xi)+\chi_{2m}(\xi) \right] = 2m ,\quad m\geq 0 .
\end{gather}
However there does not seem to be any utility in employing this for our problem.
\end{remark}

To conclude this section we give what appears to be the simplest summation representation for the 
variance of the index.
\begin{proposition}
Let $ m=\lfloor n/2\rfloor $. The variance of the index has the summation representations
\begin{align}
 \Delta_{n} &=\frac{n}{4}-\frac{n}{2\pi^2}\sum_{l=0}^{m}\frac{\Gamma^2(l+\tfrac{1}{2})}{\Gamma^2(l+1)}
                                  \left[\psi(m+\tfrac{1}{2})-\psi(m+1)-\psi(l+\tfrac{1}{2})+\psi(l+1)\right] ,
\\
  & = \frac{n}{2\pi^2}\left\{ \left[ \psi(m+1)-\psi(m+\tfrac{1}{2}) \right]\sum_{l=0}^{m}\frac{\Gamma^2(l+\tfrac{1}{2})}{\Gamma^2(l+1)}
                                 + \sum_{l=m+1}^{\infty}\frac{\Gamma^2(l+\tfrac{1}{2})}{\Gamma^2(l+1)}\left[ \psi(l+1)-\psi(l+\tfrac{1}{2}) \right] \right\} ,
\label{VoIsum}
\end{align}
for all $ n \geq 0 $.
\end{proposition}
\begin{proof}
This follows by combining (\ref{auxS}) and (\ref{auxF}), and the functional equation of the digamma function $ \psi(x) $ (see 5.5.2 of \cite{DLMF}).
The second form is derived using the identity
\begin{equation}
\sum_{l=0}^{\infty}\frac{\Gamma^2(l+\tfrac{1}{2})}{\Gamma^2(l+1)}\left[\psi(l+1)-\psi(l+\tfrac{1}{2})\right] = \frac{\pi^2}{2} .
\end{equation}
\end{proof}

\section{Integral Representations}

We now proceed to derive a number of integral representations which while in and of themselves are of
limited utility they will offer a pathway towards an asymptotic formula. 

\begin{proposition}
Let $ m = {\lfloor n/2\rfloor} $ with $ n\geq 0 $. The variance of the index has the quasi-Beta integral representation
\begin{equation}
 \Delta_{n} = \frac{n}{4}\frac{2}{\pi^3}\int_{0}^{1}dt\, t^{-1/2}(1-t)^{-1}{\rm K}(\sqrt{1-t})\left\{ 2[\psi(m+1)-\psi(m+\tfrac{1}{2})](1-t^m)-t^m\log(t) \right\} ,
\label{betaInt}
\end{equation}
where $ {\rm K}(t) $ is the complete elliptic integral of the second kind (see 13.8 of \cite{EMOT_II}),
\begin{equation}
  {\rm K}(z) = \frac{\pi}{2}{{}_{2}F_{1}}(\tfrac{1}{2},\tfrac{1}{2};1;z^2) ,
\end{equation}
\end{proposition}
\begin{proof}
We define an auxiliary function
\begin{equation}
   J(x) := \frac{\Gamma^2(x+1)}{\Gamma^2(x+\tfrac{1}{2})}\sum^{m-1}_{l=0}\frac{\Gamma^2(x-m+\tfrac{1}{2}+l)}{\Gamma^2(x-m+1+l)} ,\quad \Re(x) > -\tfrac{1}{2}, m \geq 0 ,
\label{JAUXdef}
\end{equation}
which serves as a stepping-stone to our goal because of the fact
\begin{equation}
   \Delta_{n} = \frac{n}{4}\left\{ 1+\frac{1}{\pi^2}\frac{\Gamma^2(m+\tfrac{1}{2})}{\Gamma^2(m+1)}\left. \frac{\partial}{\partial x}J \right|_{x=m} \right\} .
\end{equation} 
Now we can employ the integral representation for the square of the ratio of Gamma functions given by 
Eq. (5.50), p. 198 of \cite{TMT_Oberhettinger}, which states for $ \beta=\alpha=x-m+\frac{1}{2}, \delta=\gamma=x-m+1 $
\begin{equation}
  \frac{\Gamma^2(x-m+\tfrac{1}{2}+l)}{\Gamma^2(x-m+1+l)} = \frac{2}{\pi}\int^{1}_{0}dt\, t^{x-m-1/2+l}{\rm K}(\sqrt{1-t}) .
\end{equation} 
This is valid because $ \Re(l) > -\Re(\frac{1}{2}+x-m) $ and $ \Re(2\gamma-2\alpha)=1>0 $.
Interchanging the summation and integration then gives
\begin{equation}
  J(x) = \frac{\Gamma^2(x+1)}{\Gamma^2(x+\tfrac{1}{2})}\frac{2}{\pi}\int^{1}_{0}dt\, t^{x-m-1/2}\frac{1-t^{m}}{1-t}{\rm K}(\sqrt{1-t}) ,
\end{equation} 
where we continue to assume $ \Re(x-m)>-\frac{1}{2} $.
Employing this into our formula for $ \Delta_{n} $ and the identity
\begin{equation}
\int_{0}^{1}dt\, t^{-1/2}(1-t)^{-1}\log(t){\rm K}(\sqrt{1-t})=-\frac{\pi^3}{2} ,
\end{equation}
we deduce (\ref{betaInt}).
\end{proof}

Let us define a function similar to that of (\ref{JAUXdef})
\begin{equation}
   J_{m} := \int^{1}_{0}dt\, t^{-1/2}\frac{1-t^{m}}{1-t}{\rm K}(\sqrt{1-t}), \quad \Re(m)>0 .
\label{Jdef}
\end{equation}
One can then verify
\begin{equation}
  \Delta_{n} = \frac{n}{\pi^3}[\psi(m+1)-\psi(m+\tfrac{1}{2})]J_{m}+\frac{n}{2\pi^3}\frac{d}{dm}J_{m}
\label{VoIasy}
\end{equation}
holds where $ m = {\lfloor n/2\rfloor} $ and $ n\geq 0 $ by using the definition of $ J_{m} $ and (\ref{betaInt}).

\begin{proposition}
Again let $ m = {\lfloor n/2\rfloor} $ and $ n\geq 0 $.
The variance of the index has either of the Mellin-Barnes integral representations
\begin{equation}
  \Delta_{n} = \frac{n}{4\pi^2}\int^{c+i\infty}_{c-i\infty}\frac{ds}{2\pi i}
  \left\{ 2\left[ \psi(m+1)-\psi(m+\tfrac{1}{2}) \right]\frac{\Gamma^2(\tfrac{1}{2}-s)}{\Gamma^2(1-s)}\left[ \psi(m+s)-\psi(s) \right]
          + \frac{\Gamma^2(m+\tfrac{1}{2}-s)}{\Gamma^2(m+1-s)}\psi'(s) \right\} ,
\label{IntMellin:a}
\end{equation}
or
\begin{equation}
  \Delta_{n} = \frac{n}{2\pi^2}\int^{c+i\infty}_{c-i\infty}\frac{ds}{2\pi i}
  \frac{\Gamma^2(\tfrac{1}{2}-s)}{\Gamma^2(1-s)}\left\{[\psi(m+1)-\psi(m+\tfrac{1}{2})][\psi(m+s)-\psi(s)]+[\psi(\tfrac{1}{2}-s)-\psi(1-s)]\psi(m+s) \right\} ,
\label{IntMellin:b}
\end{equation}
where $ 0 < c < \tfrac{1}{2} $.
\end{proposition}
\begin{proof}
From the tables of \cite{TMT_Oberhettinger} we find two inverse Mellin transform formulae
\begin{align}
   \frac{x^{m}-1}{x-1} & = \int^{c+i\infty}_{c-i\infty}\frac{ds}{2\pi i}x^{-s}\left[ \psi(m+s)-\psi(s) \right], \quad c>0, \Re(m)>0 ,
\label{MXfm:a}
\\
   \frac{\log(x)}{x-1} & = \int^{c+i\infty}_{c-i\infty}\frac{ds}{2\pi i}x^{-s}\psi^{'}(s), \quad c>0 .
\label{MXfm:b}
\end{align}
We insert these into (\ref{betaInt}) and interchange the orders of the integration and are left with
the following type of integral, which is evaluated using Eq. (14.2), p. 155 of \cite{TMT_Oberhettinger}
\begin{equation}
   \int^{1}_{0}dt\, t^{s-1}{\rm K}(\sqrt{1-t}) = \frac{\pi}{2}\frac{\Gamma^2(s)}{\Gamma^2(s+\tfrac{1}{2})}, \quad \Re(s)>0 .
\end{equation}
By employing this integral we are obliged to keep $ c $ within the analyticity strips $ (0,\frac{1}{2}) $ and 
$ (0,m+\frac{1}{2}) $ for the first and second terms respectively.
This yields (\ref{IntMellin:a}) whereas to deduce (\ref{IntMellin:b}) we integrate the second term of (\ref{IntMellin:a})
by parts (which is actually valid for $ 0<c<m+\frac{1}{2} $) and displace the contour under $ s\mapsto s+m $.
\end{proof}

\begin{remark}
Employing the standard Mellin transform techniques as expounded in, say \cite{paris+Kaminski}, to extract the 
$ n\to \infty $ asymptotics from either (\ref{IntMellin:a}) or (\ref{IntMellin:b}) does not yield a clean
large $ n $ expansion. 
\end{remark}

We give a third integral representation which bears some resemblance to the Schwinger proper-time 
integral representations in Quantum Electrodynamics of particle-antiparticle systems subject to large
uniform background electromagnetic fields, see Eq. (5.25) of \cite{DR_1985} or Eq. (7.88) of \cite{GR_2009}.
\begin{proposition}
The variance of the index has an integral representation on $ (0,\infty) $ with
\begin{equation}
  J_{m} = 2\pi \int^{\infty}_{0} \frac{dx}{\sinh x} \left\{ 1-\frac{(\tfrac{1}{2})_{m}}{m!}\cosh^{-2m}x\, {}_2F_1(\tfrac{1}{2},m;m+1;\cosh^{-2}x ) \right\} ,\quad m\geq 0 .
\label{Jint}
\end{equation} 
\end{proposition}
\begin{proof}
We employ some elementary changes of variable and the elliptic integral definition of $ {\rm K}(z) $
in order to rewrite the integral as
\begin{equation}
   \int^{\infty}_{-\infty} dy\frac{1}{\sqrt{1+y^2}}\int^{\infty}_{-\infty}dx\, (x^2+y^2)^{-1}\left[ 1-\left(1+x^2+y^2\right)^{-m} \right] .
\end{equation} 
We evaluate the inner integral in terms of the Gauss hypergeometric function
$ \Re(y)\neq 0 $, $ m \in \mathbb{Z} $, $ m>0 $
\begin{equation}
\int_{-\infty}^{\infty}dx\,\frac{1}{x^2+y^2}\left[ 1-\left(1+x^2+y^2\right)^{-m} \right]
=\frac{(-1)^m\pi^{3/2}}{\Gamma(m)\Gamma(\frac{3}{2}-m)}y^{-2}\left(1+y^2\right)^{1/2-m}
 {}_{2}{\rm F}_{1}(\tfrac{1}{2},1;\tfrac{3}{2}-m;1+y^{-2}) ,
\end{equation}
and employ the connection formula Eq. (15.10.29) of \cite{DLMF} to rewrite the hypergeometric function.
The final result is (\ref{Jint}).
\end{proof}

\section{Asymptotics as $ n\to \infty $}

None of the above integral representations yield the asymptotic behaviour of the variance of
the index as $ n \to \infty $ in a particularly clean way and we therefore present an evaluation
which while being exact also constitutes an asymptotic formula.

\begin{proposition}
The function $ J_{m} $ has the evaluation
\begin{equation}
  J_{m} = \tfrac{1}{2}\pi\left[\psi(m+\tfrac{1}{2})-\psi(\tfrac{1}{2})\right]+\tfrac{1}{2}\pi\log(4)-\frac{\pi}{4(2m+1)}{}_{4}F_{3}(1,1,\tfrac{3}{2},\tfrac{3}{2};2,2,\tfrac{3}{2}+m;1) ,
\label{Idef}
\end{equation}
for $ \Re(m) > 0 $. Then the variance of the index is given by (\ref{VoIasy}).
\end{proposition}
\begin{proof}
Into the definition (\ref{Jdef}) we employ the series expansion for $ {\rm K}(z) $ about $ z=0 $, interchange the integration and 
summation and integrate term-by-term after separating out the first term from the rest. For the first
term we employ the direct Mellin transform corresponding to (\ref{MXfm:a}) and the standard beta-integral for the
remainder. In this result we recognise the series expansions of the hypergeometric functions
$ {}_3F_2 $ and $ {}_4F_3 $ with argument unity and recall the identity
\begin{equation}
{}_{3}F_{2}(\tfrac{3}{2},1,1;2,2;1) = 2 \log(4) .
\end{equation}
This gives the expression (\ref{Idef}).
\end{proof}

Clearly (\ref{VoIasy}) and (\ref{Idef}) give the large $ n $ behaviour in an explicit way with the leading
logarithmic term coming from $ \psi(m+\tfrac{1}{2}) $, which we detail in the following corollary by giving the leading order terms.
\begin{corollary}
The leading order terms of the asymptotic expansions of $ \Delta_n $ as $ n\to \infty $ are, in the
even case $ n=2k $
\begin{multline}
\Delta_{2k} = \frac{\log{8k}+\gamma+1}{2\pi^2}
-\frac{\log{8k}+\gamma}{8\pi^2 k}+\frac{7}{384\pi^2 k^2}
\\
+\frac{24\log{8k}+24\gamma-41}{1536\pi^2 k^3}-\frac{219}{81920\pi^2 k^4}
-\frac{2560\log{8k}+2560 \gamma-6247}{327680\pi^2 k^5}+\frac{19129}{16515072\pi^2 k^6}
+{\rm O}(k^{-7}\log{k}) ,
\end{multline}
and the odd case $ n=2k+1 $
\begin{multline}
\Delta_{2k+1} = \frac{\log{8k}+\gamma+1}{2\pi^2}
+\frac{\log{8k}+\gamma+2}{8\pi^2 k}
-\frac{24\log{8k}+24\gamma-7}{384\pi^2 k^2}
\\
+\frac{8\log{8k}+8\gamma-9}{512\pi^2 k^3}
+\frac{1920\log{8k}+1920\gamma-3937}{245760\pi^2 k^4}
\\
-\frac{2560\log{8k}+2560\gamma-5809}{327680\pi^2 k^5}
-\frac{322560\log{8k}+322560\gamma-882767}{82575360\pi^2 k^6}
+{\rm O}(k^{-7}\log{k}) .
\end{multline}
\end{corollary}

The asymptotics for the variance of the index can also be studied via the Birkhoff-Trjitzinsky theory
\cite{Ad_1928}, \cite{Ad_1928a}, \cite{Bi_1930}, \cite{BT_1933} for 
linear difference equations. However this is unnecessary in our case because of the following fact.
The homogeneous forms of the linear second-order difference equations (\ref{HDE:a},\ref{HDE:b}) are
\begin{equation}
  \frac{(2m+1)}{(2m-1)}X_{m+1}-\left[1+\frac{(2m+2)(2m+1)}{(2m)(2m-1)}\right]X_{m}+\frac{(2m)}{(2m-2)}X_{m-1} = 0 ,
\end{equation}
and
\begin{equation}
  \frac{(2m+1)}{(2m+3)}X_{m+1}-\left[1+\frac{(2m)(2m-1)}{(2m+2)(2m+1)}\right]X_{m}+\frac{(2m)}{(2m+2)}X_{m-1} = 0, 
\end{equation} 
which have the general solutions
\begin{equation}
 X_{m} = C_1(m)m+C_2(m) m\left[ \log{4}+\psi(m+\tfrac{1}{2})-\psi(m+1) \right]
\end{equation} 
\begin{equation}
 X_{m} = C_1(m)(m+\tfrac{1}{2})+C_2(m)(m+\tfrac{1}{2}) \left[ \log{4}+\psi(m+\tfrac{1}{2})-\psi(m+1) \right] ,
\end{equation}
respectively, for arbitrary periodic functions $ C_1(m), C_2(m) $ with period unity. 
Thus the asymptotic behaviour is clearly apparent in our simple solutions.

\section{Acknowledgments}
This research was supported by the Australian Research Council.

\bibliographystyle{plain}
\bibliography{moment,random_matrices,nonlinear}

\end{document}